\documentclass[12pt]{amsart}
\usepackage{amssymb}

\newtheorem{theorem}{Theorem}
\newtheorem{prop}{Proposition}
\theoremstyle{definition}
\newtheorem{defin}{Definition}
\newtheorem{case}{Case}
\newtheorem{subcs}{\hspace*{1em}}
\theoremstyle{remark}
\newtheorem*{erem}{Remark}

\def\al{\alpha}	\def\be{\beta}	
\def\ga{\gamma}	\def\la{\lambda}
\def\fK{\mathbf k}	\def\gM{\mathfrak m}
\def\fV{\mathbf v}	\def\rB{\mathrm B}

\def\sb{\subset}	\def\sbe{\subseteq}
\def\sp{\supset}	\def\spe{\supseteq}

\def\set#1{\left\{\,#1\,\right\}}
\def\setsuch#1#2{\left\{\,#1\mid #2\,\right\}}
\def\gnr#1{\langle\,#1\,\rangle}
\def\lst#1#2{ #1_1 , #1_2 , \dots , #1_{#2} }

\def\pa{\mathop\mathrm{par}\nolimits}
\def\End{\mathop\mathrm{End}\nolimits}
\def\rad{\mathop\mathrm{rad}\nolimits}
\def\chr{\mathop\mathrm{char}\nolimits}
\def\Gr{\mathop\mathrm{Gr}\nolimits}
\def\hS{\bar S}	\def\hE{\bar E}	\def\hW{\bar W}
\def\tI{\tilde I}
\def\dm{\mbox{-}}

\def\iff{if and only if }

\title[Ideals of one branch curve singularities]{One branch curve singularities with at 
most 2-parameter families of ideals}

\author{Yuriy A. Drozd}
\address{Institute of Mathematics, National Academy of Sciences of Ukraine}
\email{y.a.drozd@gmail.com, drozd@imath.kiev.ua}
\urladdr{www.imath.kiev.ua/$\sim$drozd}

\author{Ruslan V. Skuratovskii}
\address{Kyiv National Taras Shevchenko University}
\email{ruslcomp@mail.ru}

  \subjclass[2010]{Primary 13C05, Secondary 14H20}
\keywords{Curve singularity, ideal, family of ideals, sandwich technique}

\begin{document}

\begin{abstract}
 A criterion is given in order that the ideals of a one branch curve 
singularity form at most $2$-parameter families. Namely, we present a list of plane curve 
singularities from the Arnold's classification which are the smallest among all one 
branch singularities having at most $2$-parameter families of ideals.
\end{abstract}

\maketitle

\section*{Introduction}

 Ideals of commutative rings have been studied at least since the works of Dedekind on 
the ideals of algebraic numbers. The Dedekind domains, i.e. integrally closed noetherian 
domains of Krull dimension $1$, are just domains such that all their ideals are 
invertible. If a domain $A$ is not integrally closed, the theory of 
ideals becomes rather complicated. As it was noticed by Bass \cite{ba0} and, 
independently, by Borevich and Faddeev \cite{bf}, if $A$ is of Krull dimension $1$ and 
its integral closure $R$ has $2$ generators as $A$-module, every ideal is invertible over 
its multiplication ring (and vice versa). Moreover, in this case all finitely generated 
torsion free $A$-modules are direct sums of ideals. Jacobinski \cite{ja} and, 
independently, Drozd and Roiter \cite{dr} gave criteria for a commutative ring of Krull 
dimension $1$ to have finitely many nonisomorphic torsion free modules. It so happens  
that it is also the case when it has finitely many ideal classes. As Greuel and 
Kn\"orrer showed, in the local case these are just the rings dominating the \emph{simple 
plane curve singularities} $A\dm D\dm E$ of Arnold \cite{avg}. Schappert \cite{sch}, 
Drozd and Greuel \cite{dg} proved that a local ring of plane curve singularity only has 
$1$-parameter families of ideals \iff this singularity dominates a \emph{strictly 
unimodal} plane curve singularity (in \cite{avg} they are called \emph{unimodal} and 
\emph{bimodal}). This time it is no more the case that torsion free modules behave in the 
same manner. Among strictly unimodal plane curve singularities only those of type 
$T_{pq}$ are \emph{tame}, i.e. only have $1$-parameter families of indecomposable torsion 
free modules \cite{dg1}. All others are \emph{wild}, so have $n$-parameter families of 
nonisomorphic indecomposable torsion free modules for arbitrary $n$.

 In this paper we find a criterion for a one branch curve singularity to have at most 
$2$-parameter families of ideals. It so happens that such singularities can also be 
characterized using the Arnold lists from \cite[Section 15.1]{avg}. Namely, they are just 
those dominating one of the singularities of type 
$ E_{30},\, E_{32},\, W_{24},\, W^\sharp_{2,*},\, W_{30},\, N_{20},\, N_{24}$ or $N_{28}$ 
(see Theorem \ref{t1}). To prove this result we use the ``sandwich'' technique, just as 
in the papers cited above. Certainly, the ``one branch'' condition is rather restrictive 
and one would like to get rid of it, but even in this case the calculations are 
cumbersome, so we had to restrict our ambition.

\section{Main Theorem}

 We fix an algebraically closed field $\fK$.

\begin{defin}\label{d0}
  A \emph{one branch curve singularity} is a complete local noetherian $\fK$-algebra $S$ of Krull dimension $1$ without zero divisors and such that $S/\gM=\fK$, where $\gM$ is the maximal ideal of $S$. It is called \emph{plane} if $\gM$ is generated by $2$ elements. 
\end{defin}

 Such an algebra is indeed isomorphic to the completion of a local ring of a (singular) point $p$ of an algebraic curve $X$ over $\fK$; this curve can be chosen plane if so is the singularity. Moreover, the curve $X$ is irreducible in the formal neighbourhood of the point $p$, or, the same, $p$ belongs to a unique branch (place, formal component) of $X$ in the sense of 
\cite{gri} or \cite{wa}.

 It is known that the \emph{normalization} of a one branch curve singularity $S$ is 
isomorphic to the algebra $R=\fK[[t]]$ of formal power series and $R$ is finitely generated as $S$-module. So we always suppose that $t^rR\sb S\sb R$ for some $r$. For every element $x\in R$ let $v(x)$ be its \emph{valuation}, i.e. $x=t^{v(x)}u$, where $u$ is invertible in $R$. If $S$ is plane and $\gM=(x,y)$, one can always suppose that $v(x)<v(y)$ and $v(x)\nmid v(y)$. Then we call the pair $\fV=(v_1,v_2)$ the 
\emph{valuation vector} of $S$. Obviously, it does not depend on the choice of such generators. Note that every plane curve singularity is Gorenstein \cite{ba}. 

 We recall the definition of the \emph{parameter number of ideals} $\pa(S)=\pa(1,S)$ from \cite{dg0,dg}. Remark first that every ideal of $S$ is isomorphic to an $S$-submodule $M\sbe R$ containing $S$ \cite{di}. Let $\rB(d)$ be the closed subset of the Grassmannian $\Gr(d,R/S)$ consisting of those spaces which are $S$-submodules. Every point $b\in\rB(d)$ gives rise 
to an $S$-submodule $M(b)$ of $R$ which contains $S$. We set
\begin{align*}
   O(b)&=\setsuch{b'\in\rB(d)}{M(b')\simeq M(b)},\\
  \rB(d,k)&=\setsuch{b\in\rB(d)}{\dim O(b)=k} \\
\intertext{and}
  \pa(S)&=\max_{d,k}\set{\dim\rB(d,k)-k}.
\end{align*}
 Note that both $O(b)$ and $\rB(d,k)$ are locally closed subsets in $\rB(d)$. Intuitively, $\pa(S)$ is the biggest possible number of independent parameters that define isomorphism classes of $S$-ideals.

\begin{defin}\label{d1}
 Let $S$ be a one branch plane curve singularity, $\fV$ be its valuation vector. We say that $S$ is
\begin{itemize}
\item  \emph{of type} $E_{6k}$ if $\fV=(3,3k+1)$,
\item  \emph{of type} $E_{6k+2}$ if $\fV=(3,3k+2)$,
\item  \emph{of type} $W_{6k}$ if $\fV=(4,2k+1)$,
\item  \emph{of type} $W^\sharp_{k,*}$ if $\fV=(4,4k+2)$,
\item  \emph{of type} $N_{4k}$ if $\fV=(5,k+1)$.
\end{itemize}
\end{defin}

\begin{erem}
  In \cite{dg} it is shown that, if $\chr\fK=0$, our definitions of singularities of 
types $E$ and $W$ are equivalent to those given in \cite[\S\,15]{avg} in terms of the 
normal forms of equations. We do not precise the equations for singularities of type $N$, 
since they are complicated and we do not use them.
\end{erem}

They say that a singularity $S'$ \emph{dominates} the singularity $S$, or is an \emph{over-ring} of $S$, if $S\sbe S'\sbe R$.

\begin{theorem}\label{t1}
 Let $S$ be a one branch curve singularity. The following conditions are equivalent:
\begin{enumerate}
\item  $\pa(S)\le 2$.
\item Either $\chr\fK\ne2$ and $S$ dominates one of the following singularities:
\[
  E_{30},\, E_{32},\, W_{24},\, W^\sharp_{2,*},\, W_{30},\, N_{20},\, N_{24},\, N_{28},  
\]
 or $\chr\fK=2$ and $S$ dominates one of the following singularities:
\[
  E_{30},\, E_{32},\, W_{18},\, W^\sharp_{1,*},\, N_{20},\, N_{24}. 
\]
\end{enumerate}
\end{theorem}
\begin{proof}

 We suppose that $\chr\fK\ne2$. If $\chr\fK=2$, the calculations are quite similar (even easier, since less cases must be considered). As usually, we denote by $\gnr{\lst am}$ the vector space (over $\fK$) generated by $\lst am$.

 Let $m=\dim R/\gM R$. It is known that also $\dim I/\gM I\le m$ for all ideals $I$ of $S$ \cite{di}. If $m>5$, then the same observations as in \cite[\S\,2.2]{dg} show that $\pa(S)\ge3$. If $m=2$, $S$ is a Bass ring \cite{di}, so has finitely many ideals up to isomorphism. For $m=3$ the result follows from \cite[Theorem 4.1]{ds}. For $m=4$ it was proven in \cite{sk}. Hence, we only have to consider the case $m=5$. Then $S$ contains an element $x$ with $v(x)=5$, so we deal with singularities of type $N$. In section 2 we will calculate the ideals of the rings of types $N_{4k}\ (k\le 7)$ and show that there are at most $2$-parameter families in these cases. Therefore, we must show now that $\pa(1,S)\ge3$ if $m=5$ and $S$ does not contain any element $y$ with $6\le v(y)\le 8$. If it is the case, $S\sbe S_0$, where $S_0=\fK+\fK x+t^9R$. The maximal ideal of $S_0$ is $\gM_0=\fK x+t^9R$. One easily checks that the $S$-ideals
 \[
  I(\al,\be,\ga)=\gnr{1,t+\al t^3+\be t^4+\ga t^8}+\gM_0,\ \text{where }\al,\be,\ga\in\fK,
 \]
 are pairwise non-isomorphic. It implies that $\pa(1,S_0)\ge3$, hence, $\pa(1,S)\ge3$ for every $S\sbe S_0$.
\end{proof}

 We recall the \emph{sandwich procedure} used for calculation of ideals \cite{di,dg}. Let $S$ be a curve singularity, $\gM=\rad S$ and $S'=\End\gM$. We consider $S'$ as an over-ring of $S$ and set $\hS=S'/\gM$. If $I$ is an $S$-ideal, then $I'=S'I$ is an $S'$-ideal and $I'\spe I\sp\gM I=\gM I'$. So $I$ is defined by the subspace $V= I/\gM I$ of the $\hS$-module $W=I'/\gM I'$. This subspace is not arbitrary, but \emph{generating} in the sense that $\hS V=W$. Let $E=\End I'$, $E_0=\setsuch{a\in E}{aI'\sbe\gM I'}$ and $\hE=E/E_0$. Then $W$ is a $\hE$-module As it was mentioned above, we can and always will suppose that $R\spe I\spe S$, thus $R\spe I'\spe S'$. Then $E\sbe I'$, so $\hE\sbe W$, and two generating subspaces $V,V'\sbe W$ define isomorphic ideals \iff $V'=aV$ for an element $a\in\hE$. Moreover, since we only consider subspaces containing the class of $1$ (which we also denote by $1$), such element $a$ belongs to $V'$. Let $\hW=I'/\gM'I'$, where $\gM'=\rad S'$. Then the subspace $V\sbe W$ is generating \iff its image in $\hW$ is the 
whole $\hW$. Therefore, if $\dim\hW=m=\dim R/\gM R$, then $\gM'I'=\gM I'$, hence the unique generating subspace of $W$ is $W$ itself, so the unique $S$-ideal $I$ with $S'I=I'$ is $I'$. In the further calculations we will not consider such $S'$-ideals at all. The case $V=W$ will also be omitted, since then $I=I'$.

%\newpage
\section{Description of ideals of singularities of type $N$}

 Since the calculations are quite analogous in all cases, we consider the ``deepest'' singularity of type $N_{28}$, when the valuation vector is $(5,8)$. So, let $S\sb R=\fK[[t]]$ be generated (as a complete local $\fK$-algebra) by 
the elements $x,y$, where $v(x)=5,\,v(y)=8$. We may suppose that $t=y^2/x^3$. We also set $z=y/x$. Then $S\sp t^{28}R$. Moreover, since $S$ is Gorenstein, every $S$-ideal is either principal or an $S_0$-ideal, where 
 \[
 S_0=\End\gM=S+\gnr{t^{27}}=\gnr{1,x,y,x^2,xy,x^3,y^2,x^2y,x^4,xy^2}+t^{23}R
 \]
 (see \cite{ba} or \cite{di}). Consider the chain of rings $S_0\sb S_1\sb S_2\sb S_3\sb S_4$, where
\begin{align*}
  S_1=\End\gM_0&=\gnr{1,	x,y,x^2,xy,x^3,y^2}+t^{18}R,\\
  S_2=\End\gM_1&=\gnr{1,x,y,x^2,tx^2}+t^{13}R,\\
 S_3=\End\gM_2&=\gnr{1,z,x}+t^8R,\\
 S_4=\End\gM_3&=\gnr{1,z}+t^5R,
\end{align*}
 and $\gM_i=\rad S_i$. The $S_3$-ideals are known \cite{ja,dr,ds}; they are (except $R,S_4$ and $S_3$ itself):
\begin{align*}
  R_2&=\gnr{1}+t^2R,\\
  R_3&=\gnr{1}+t^3R,\\
  R_3^*&=\gnr{1,t}+t^3R,\\
  S_4^*&=\gnr{1,t^2,t^3}+t^5R.
\end{align*}
 The ideals $R_3^*$ and $S_4^*$ are indeed dual to $R_3$ and $S_4$ respectively, though we will not use this property. Note that it follows from \cite{di} that every $S_3$-ideal is isomorphic either to an over-ring of 
$S_3$ or to the ideal dual to such an over-ring.

 \begin{prop}\label{S2}
 Here is a complete list of representatives of the ideal classes of the ring $S_2$, which are not $S_3$-ideals, sorted by the induced $S_3$-ideals $I'=S_3I$. We set $\tI=\gM_2I'$.
\begin{itemize}
\item   $I'=S_3;\ \tI=\gM_2$:
	\begin{enumerate}
	\item $S_2$,
	\item  $F_1(\al,\be)=\gnr{1,z+\al z^3+\be z^4}+\gM_2$,
	\item  $F_2(\al)=\gnr{1,z^2+\al z^3}+\gM_2$,
	\item  $F_3(\al)=\gnr{1,z^3+\al z^4}+\gM_2$,
	\item  $I_1=\gnr{1,z^4}+\gM_2$,
	\item  $F_4(\al,\be)=\gnr{1,z+\al z^4,z^2+\be z^4}+\gM_2$,
	\item  $F_5(\al)=\gnr{1,z,z^3+\al z^4}+\gM_2$,
	\item  $F_6(\al)=\gnr{1,z+\al z^3,z^4}+\gM_2$,
	\item  $F_7(\al)=\gnr{1,z^2,z^3+\al z^4}+\gM_2$,
	\item  $F_8(\al)=\gnr{1,z^2+\al z^3,z^4}+\gM_2$,
	\item  $I_2=\gnr{1,z^3,z^4}+\gM_2$,
	\item  $I_3=\gnr{1,z.z^2,z^3}+\gM_2$,
	\item  $I_4=\gnr{1,z,z^2,z^4}+\gM_2$,
	\item  $I_5=\gnr{1,z,z^3,z^4}+\gM_2$,
	\item  $I_6=\gnr{1,z^2,z^3,z^4}+\gM_2$.
	\end{enumerate}

\item  $I'=S_4; \ \tI=\gnr{x,y}+t^{10}R$:
	\begin{enumerate}
	\item $F_9(\al,\be)=\gnr{1,z+\al tz^2+\be z^3}+\tI$, where $\al\ne0$,
	\item  $F_{10}(\al,\be)=\gnr{1,z^2+\al tz^2+\be z^3}+\tI$, where $\al\ne0$,
	\item  $F_{11}(\al)=\gnr{1,tz^2+\al z^3}+\tI$,
	\item  $F_{12}(\al,\be)=\gnr{1,z+\al tz^2,z^2+\be tz^2}+\tI$,\\ where $\al\ne0$ or $\be\ne0$,
	\item  $F_{13}(\al,\be)=\gnr{1,z+\al z^3,tz^2+\be z^3}+\tI$,
	\item  $F_{14}(\al)=\gnr{1,z+\al tz^2,z^3}+\tI$, where $\al\ne0$,
	\item  $F_{15}(\al,\be)=\gnr{1,z^2+\al z^3,tz^2+\be z^3}+\tI$,
	\item  $F_{16}(\al)=\gnr{1,z^2+\al tz^2, z^3}+\tI$, where $\al\ne0$,
	\item  $I_7=\gnr{1,tz^2, z^3}+\tI$,
	\item  $F_{17}(\al)=\gnr{1,z,z^2,tz^2+\al z^3}+\tI$,
	\item  $F_{18}(\al,\be)=\gnr{1,z+\al tz^2,z^2+\be tz^2, z^3}+\tI$,\\
	 where $\al\ne0$ or 	$\be\ne0$,
	\item  $I_8=\gnr{1,z,tz^2, z^3}+\tI$,
	\item  $I_9=\gnr{1,z^2,tz^2, z^3}+\tI$.
	\end{enumerate}
	
\item  $I'=S_4^*; \ \tI=\gnr{x,tz^2,y}+t^{10}R$:
	\begin{enumerate}
	\item $F_{19}(\al,\be)=\gnr{1,t^2+\al z+\be z^2}+\tI$,
	\item  $F_{20}(\al,\be)=\gnr{1,t^2+\al z^2, z+\be z^3}+\tI$,
	\item  $F_{21}(\al,\be)=\gnr{1,t^2+\al z,z^2+\be z^3}+\tI$,
	\item  $F_{22}(\al,\be)=\gnr{1,t^2+\al z+\be z^2,z^3}+\tI$,
	\item  $F_{23}(\al)=\gnr{1,t^2,z+\al z^3,z^2}+\tI$,
	\item  $F_{24}(\al)=\gnr{1,t^2+\al z,z^2, z^3}+\tI$, 
	\item  $F_{25}(\al)=\gnr{1,t^2+\al z^2,z, z^3}+\tI$. 
	\end{enumerate}
	
\item  $I'=R_3; \ \tI=\gnr{x}+t^8R$:
	\begin{enumerate}
	\item $F_{26}(\al,\be)=\gnr{1,z+\al tz+\be tz^2}+\tI$, where $\al\ne0$,
	\item  $F_{27}(\al,\be)=\gnr{1,tz+\al z^2+\be tz^2}+\tI$, 
	\item  $F_{28}(\al,\be)=\gnr{1,z,tz+\al z^2+\be tz^2}+\tI$,
	\item  $F_{29}(\al,\be)=\gnr{1,z+\al tz,z^2+\be tz^2}+\tI$, where $\al\ne0$,
	\item  $F_{30}(\al)=\gnr{1,z+\al tz,tz^2}+\tI$, where $\al\ne0$,
	\item  $F_{31}(\al,\be)=\gnr{1,tz+\al tz^2,z^2+\be tz^2}+\tI$,
	\item  $F_{32}(\al)=\gnr{1,tz+\al z^2,tz^2}+\tI$,
	\item  $F_{33}(\al)=\gnr{1,z,tz,z^2+\al tz^2}+\tI$, where $\al\ne0$,
	\item  $I_{10}=\gnr{1,z,tz,tz^2}+\tI$,
	\item  $F_{34}(\al)=\gnr{1,z+\al tz,z^2,t z^2}+\tI$, where $\al\ne0$,
	\item  $I_{11}=\gnr{1,z,z^2,tz^2+\al z^3}+\tI$.
	\end{enumerate}

\item  $I'=R_3^*; \ \tI=\gnr{x,z^2}+t^8R$:
	\begin{enumerate}
	\item  $F_{35}(\al,\be)=\gnr{1,t+\al z+\be tz}+\tI$, where $\al\ne0$,
	\item  $F_{36}(\al,\be)=\gnr{1,t,z+\al tz+\be tz^2}+\tI$, 
	\item  $F_{37}(\al,\be)=\gnr{1,t+\al z,tz+\be tz^2}+\tI$,
	\item  $F_{38}(\al,\be)=\gnr{1,t+\al z+\be tz, tz^2}+\tI$, 
	\item  $F_{39}(\al)=\gnr{1,t,z+\al tz^2,tz}+\tI$, 
	\item  $F_{40}(\al)=\gnr{1,t,z+\al tz, tz^2}+\tI$,
	\item  $F_{41}(\al)=\gnr{1,t+\al z,tz,tz^2}+\tI$.
	\end{enumerate}

\item  $I'=R_2; \ \tI=\gnr{x}+t^7R$:
	\begin{enumerate}
	\item  $F_{42}(\al,\be)=\gnr{1,t^2+\al z^2,z+\be tz}+\tI$, where $\be\ne0$,
	\item  $F_{43}(\al,\be)=\gnr{1,t^2+\al z, tz}+\tI$, 
	\item  $I_{12}(\al,\be)=\gnr{1,t^2,z,tz}+\tI$,
	\item  $F_{44}(\al,\be)=\gnr{1,t^2,z+\al tz, z^2}+\tI$, where $\al\ne0$,
	\item  $F_{45}(\al)=\gnr{1,t^2+\al z,tz,z^2}+\tI$. 
	\end{enumerate}

\item  $I'=R; \ \tI=t^5R$:
	\begin{enumerate}
	\item  $F_{46}(\al,\be)=\gnr{1,t+\al t^4,t^2+\be t^4}+\tI$, 
	\item  $I_{13}=\gnr{1,t,t^2, t^3}+\tI$, 
	\item  $I_{14}=\gnr{1,t,t^2,t^4}+\tI$.
	\end{enumerate}
\end{itemize}
 In all these formulae $\al$ and $\be$ denote some elements from the field $\fK$. 
Moreover, all quotient spaces $W=I'/\gM_2 I'$ are of dimension $5$, so all $S$-ideals $I$
such that $S_2I=I'$ are actually $S_2$-ideals. Therefore, we need not consider them in 
the further calculations.
\end{prop}
\begin{proof}
  We only consider the case when $I'=S_4$, since all other cases are quite similar 
(mostly easier). Then $W=\gnr{1,z,z^2,tz^2,z^3}$, where we denote the class of an element by the same symbol as the element itself, and $\hW=\gnr{1,tz^2}$. Therefore, the dimension of a generating subspace $V$ is at least $2$. We also suppose that $1$ is an element of a basis of $V$.

If $\dim V=2$, there are the following possibilities:
 \begin{enumerate}
\item $V=\gnr{1,v}$, where $v=z+\ga z^2+\al tz^2+\eta z^3$ and $\al\ne0$, since $V$ must project onto $\hW$. Set $\be=\eta-\ga^2$ and $a=1-tu$, where $u=z+\al tz^2+\be z^3$. Then $aV=\gnr{1,u}$ and the preimage of $aV$ in $S_4$ is $F_9(\al,\be)$ from the list. On the other hand, the image of $F_9(\al,\be)$ in $W$ is $V(\al,\be)=\gnr{1,z+\al tz^2+\be z^3}$. If $V(\al',\be')=aV(\al,\be)$, then $a\in V(\al',\be')$, so we may suppose that $a=1-\la(z+\al'tz^2+\be'z^3)$. Then the condition $a(z+\al tz^2+\be z^3)\in V(\al',\be')$ implies that $\al'=\al$ and $\be'=\be$. Therefore, the ideals $F_9(\al,\be)$ are pairwise nonisomorphic. Further on we omit such verifications of nonisomorphy, since they are easy and straightforward.

\item  $V=\gnr{1,z^2+\al tz^2+\be z^3}$ gives rise to $F_{10}(\al,\be)$. Again one easily checks that all these ideals are nonisomorphic.

\item  $V=\gnr{1,tz^2+\al z^3}$ gives rise to $F_{11}(\al)$.
\end{enumerate}

If $\dim V=3$, there are the following possibilities:
\begin{enumerate}\setcounter{enumi}3
\item  $V=\gnr{1,u,v}$, where $u=z+\al tz^2+\ga z^3,\,v=z^2+\be tz^2+\eta z^3$ with $\al\ne0$ or $\be\ne0$. Let  $a=1-\eta u'-\ga v'$, where $u'=z+\al tz^2,\,v'=z^2+\be tz^2$; then $aV=\gnr{1,u',v'}$, so gives rise to $F_{12}(\al,\be)$.

\item  $V=\gnr{1,u,v}$, where $u=z+\ga z^2+\eta z^3,\,v=tz^2+\be z^3$. Set $\al=\eta-\ga^2$ and $u'=z+\al z^3$. Then $(1-\ga u')V=\gnr{1,u',v}$, so it gives rise to $F_{13}(\al,\be)$.

\item  In the same way $V=\gnr{1,z+\be z^2+\al tz^2,z^3}$ is reduced to $\gnr{1,z+\al tz^2,z^3}$ and gives rise to $F_{14}(\al)$.

\item  $V=\gnr{1,z^2+\al z^3,tz^2+\be z^3}$ gives rise to $F_{15}(\al,\be)$.

\item  $V=\gnr{z^2+\al tz^2,z^3}$ gives rise to $F_{16}(\al)$.

\item  $V=\gnr{z^2,tz^2,z^3}$ gives rise to $I_7$.
\end{enumerate}

Finally, if $\dim V=4$, there are the following possibilities:
\begin{enumerate}\setcounter{enumi}9
\item  $V=\gnr{1,u,v,w}$, where $u=z+\be z^3,\,v=z^2+\ga z^3,\,w=tz^2+\al z^3$. Then $(1-\ga z-\be z^2)V=\gnr{1,z,z^2,tz^2+\al z^3}$ and gives rise to $F_{17}(\al)$.

\item  $V=\gnr{1,z+\al tz^2,z^2+\be tz^2,z^3}$ gives rise to $F_{18}(\al,\be)$.

\item  $V=\gnr{1,z+\al z^2,tz^2,z^3}$. Then $(1-\al z)V=\gnr{1,z,tz^2,z^3}$ and gives rise to $I_8$.

\item  $V=\gnr{1,z^2,tz^2,z^3}$ gives rise to $I_9$.
\end{enumerate}
\end{proof}

 Now we have to find, for every $S_2$-ideal $I$ from this list, all $S_1$-ideals $I^1$ 
such that $S_2I^1=I$, then to find, for every $I^1$,  all $S_0$-ideals $I^0$ such that 
$S_1I^0=I^1$ or $S_1I^0=I$. Since the calculations are quite similar for all ideals $I$ 
and very much alike the calculations from the preceding proof (even easier), we only 
present several ``typical'' cases.
 
 \begin{case}
 (This case is the most complicated.)

  $I=S_2$, $\gM_1S_2=\gM_0S_2=\gM_1$.
 
 It gives new $S_1$-ideals:
\begin{enumerate}
 \item  
  $S_1$,
 \item  
  $I^1_1(\al,\be)=\gnr{1,u(\al,\be)}+\gM_1$,  where $u(\al,\be)
=z^2x(1+\al z+\be z^2)$,  
 \item  
  $I^1_2(\al)=\gnr{1,z^3x+\al z^4x}+\gM_1$,
 \item  
  $I^1_3=\gnr{1,z^4x}+\gM_1$,
 \item  
 $I^1_4(\al,\be)=\gnr{1,z^2x+\al z^4x,z^3x+\be z^4x}+\gM_1$,
 \item  
 $I^1_5(\al)=\gnr{1,z^2x+\al z^3x,z^4x}+\gM_1$,
 \item  
 $I^1_6=\gnr{1,z^3x,z^4x}+\gM_1$.
\end{enumerate}

 $S_1/\gM_0=\gnr{1,t^{19},t^{22}}$. It gives new $S_0$-ideals:

  $S_0,\ \gnr{1,t^{19}+\al t^{22}},\ \gnr{1,t^{22}}$. 
 
 $\gM_0I^1_1(\al,\be)=\gM_1I^1_1(\al,\be)$ if $\be\ne\al^2$. If $\be=\al^2$, then 

$I^1_1(\al,\al^2)/\gM_0I^1_1(\al,\al^2)=\gnr{1,u,t^{19}}$, which gives new $S_0$-ideals
 
 $ \gnr{1,u(\al,\al^2)+\ga t^{19},z^2xy+\al z^3xy}+\gM_0$.
\end{case}
   
\vspace*{.5ex}
\begin{case}
  $I=F_1(\al,\be)$; $\gM_1I=\gM_0I=\gnr{\al yz^2+\be yz^3,yz+\al yz^3}+\gM_1$.
\end{case}
\begin{subcs}
  $\al\ne0$. Then the only new possibilities are

 $I^1=\gnr{1,z+\al z^3+\be z^4}+\gM_1I$,

 $I^0=\gnr{1,z+\al z^3+\be z^4}+\gM_0I^1$, where

 $\gM_0I^1=\gnr{\al yz^2+\be yz^3,yz+\al yz^3+\be yz^4}+\gM_0$.
\end{subcs}
\begin{subcs}
  $\al=0,\,\be\ne0$. Then the only new possibilities are

 $I^1=\gnr{1,z+\be z^4}+\gM_1I$.

 Since $\gM_0I^1=\gM_1I^1$, no new $I^0$ occur.
\end{subcs}
\begin{subcs}
  $\al=\be=0$.  Then the only new possibilities are

 $I^1=V+\gM_1I$, where $V$ is one of the following subspaces:

 $\gnr{1,z+\ga yz^3}$, or $\gnr{1,z+\ga yz^3,yz^2+\ga' yz^3}$, or $\gnr{1,z,yz^3}$.

 In all cases $\gM_0I^1=\gM_1I^1$, so no new $I^0$ occur.
\end{subcs}

\vspace*{.5ex}
 \begin{case}
  $I=I_1$; $\gM_1I=\gM_0I=\gnr{x,y,x^2,xy}+t^{15}R$.

  The only new possibility is $I^1=\gnr{1,\al tx^2+z^4}+\gM_1 I\,$.

 Since $\gM_0 I^1=\gM_1 I^1$, so no new $I^0$ occur.
\end{case}
 
\vspace*{.5ex}
 \begin{case}
  $I=F_{11}(\al)$; $\gM_1I=\gM_0I=\gnr{x,y,x^2}+t^{13}R$.

  The only new possibility is $I^1=\gnr{1,tz^2+\al z^3+\be tx^2}+\gM_1 I$.

  Since $\gM_0 I^1=\gM_1 I^1$, so no new $I^0$ occur.
\end{case}
 
\vspace*{.5ex}
 \begin{case}
  $I=F_{20}(\al,\be)$. If $\be\ne0$, then $\gM_0I=\gM_1I=\gM_2I$, so no new $S_1$- and 
$S_0$-ideals occur. If $\be=0$, then $\gM_0I=\gM_1I=\gnr{t^2x,zy,z^4}+\gM_1$ and 
 we get new $S_1$-ideals

 $I^1(\al)=\gnr{1,t^2+\al z^2,z}+\gM_0I$.

 Again, $\gM_0I^1(\al)=\gM_1I_1(\al)$, so no new $S_0$-ideals occur.
  \end{case}

\end{document}